\documentclass[twoside, 10pt]{article}
\usepackage{amsmath,amsthm,amssymb,amscd,amstext,amsfonts}
\usepackage[utf8]{inputenc}
\usepackage{mathtools}
\usepackage{mathrsfs}
\usepackage{thmtools}
\usepackage{latexsym}
\usepackage{verbatim}
\usepackage{framed}
\usepackage{graphicx}
\usepackage{stmaryrd}
\usepackage{enumitem}
\usepackage{fullpage}
\usepackage{bm}
\usepackage{url}
\usepackage{physics}
\usepackage{dsfont}
\usepackage{todonotes}
\usepackage{tikz}
\usepackage{graphicx}
\usepackage{titling}
\usepackage{color}
\usepackage{xpatch}

\usepackage[small]{titlesec}
\titlelabel{\thetitle.\enspace}

\usepackage[breaklinks=true, linktocpage=true,]{hyperref}

\hypersetup{
	colorlinks = true,
    linkcolor={purple},
    urlcolor={purple},
    citecolor={blue!80!black}
}

\makeatletter
\def\tagform@#1{\maketag@@@{(\ignorespaces{\oldstylenums{#1}}\unskip\@@italiccorr)}}
\renewcommand{\eqref}[1]{\textup{{\normalfont(\oldstylenums{\ref{#1}}}\normalfont)}}
\let\over\@@over
\let\atop\@@atop
\makeatother

\usepackage[nameinlink, noabbrev, capitalize]{cleveref}

\usepackage[letterpaper, total={5.6in, 8in}, vmarginratio=1:1, hmarginratio=1:1, headsep=21pt]{geometry}

\declaretheoremstyle[bodyfont=\normalfont\slshape, notefont=\normalfont\itshape,notebraces={{\rm(}}{{\rm)}}, postheadspace=0.5em,headpunct={\rm.}, spaceabove=8pt, spacebelow=8pt]{slbody}

\declaretheorem[name=Theorem, numberwithin=section, style=slbody]{theorem}
\declaretheorem[name=Lemma, numberwithin=section, sibling=theorem, style=slbody]{lemma}
\declaretheorem[name=Corollary, numberwithin=section, sibling=theorem, style=slbody]{corollary}
\declaretheorem[name=Proposition, numberwithin=section, sibling=theorem, style=slbody]{proposition}
\declaretheorem[name=Conjecture, numberwithin=section, sibling=theorem, style=slbody]{conjecture}

\renewcommand\norm[1]{\left|\!\left|#1\right|\!\right|}

\newcommand\normF[1]{\left|\!\left|#1\right|\!\right|_{\rm F}}
\newcommand\bignormF[1]{\bigl|\!\bigl|#1\bigr|\!\bigr|_{\rm F}}

\newcommand\normrow[1]{\left|\!\left|#1\right|\!\right|_{\rm row}}
\newcommand\normcol[1]{\left|\!\left|#1\right|\!\right|_{\rm col}}
\newcommand\normgamma[1]{\left|\!\left|#1\right|\!\right|_{\gamma_2}}

\newcommand\normalg[1]{\left|\!\left|#1\right|\!\right|_{\rm A}}

\renewcommand\hat{\widehat}
\newcommand\T{{\rm T}}

\DeclareMathOperator*{\argmin}{arg\,min}

\DeclareMathOperator{\TD}{TD}

\newcommand{\RR}{\mathbf{R}}   
\newcommand{\ZZ}{\mathbf{Z}}   

\newcommand{\one}{\mathop{\mathbf{1}}\nolimits}
\newcommand\matr[4]{\big({#1\atop #3}{#2\atop #4}\big)}

\renewcommand{\maketitle}{%
  \begin{center}
    {\large\bf Block structure in boolean matrices of bounded factorization norm}\\
    \smallskip
    {}
    \vskip 30pt
    {\sc Marcel K.\ Goh\enspace{\rm and}\enspace Hamed Hatami}
    \medskip

    {\sl School of Computer Science, McGill University}
    \vskip 30pt
  \end{center}
}

\title{}
\author{}
\date{}

\begin{document}

\maketitle

\renewenvironment{abstract}{\quotation\noindent\small{\bfseries\abstractname.\enspace}}{\endquotation}

\begin{abstract}
A boolean matrix is blocky if its $1$-entries form  a collection of 1-monochrom\-atic submatrices that are disjoint in both rows and columns. Blocky matrices are precisely the set of boolean matrices with $\gamma_2$ factorization norm at most $1$.

Building on recent work by Balla, Hambardzumyan, and Tomon, we show that for any boolean matrix with $\gamma_2$ norm at most $\lambda$, there exists a a collection of row- and column-disjoint 1-monochromatic submatrices that together cover a significant portion (at least a $1/2^{2^{O(\lambda)}}$ fraction) of its $1$-entries. 


\end{abstract}

\vskip50pt
\baselineskip=13.5pt

\section{Introduction}

The \emph{$\gamma_2$ factorization norm} (or the \emph{$\gamma_2$ norm} for short) of a real matrix $A$ is defined to be
\begin{equation}\normgamma A = \min_{UV = A} \normrow U \normcol V,\end{equation}
where the minimum is taken over all factorizations $A = UV$,
$\normrow U$ is the maximum $\ell^2$-norm of a row in $U$, and
$\normcol V$ is the maximum $\ell^2$-norm of a column in $V$.
Note that for any submatrix $A'$ of $A$, we must have $\normgamma{A'}\ge \normgamma A$, since removing rows of $U$ or columns of $V$ cannot increase $\normrow U$ or $\normcol V$.

A topic of substantial recent interest is the degree to which structure is forced upon boolean matrices with small $\gamma_2$ norm. We shall give an abbreviated account of developments leading to our present result, referring the reader to the paper~\cite{blockypolylog2025} of the authors for a more detailed chronicle of related work.

We begin by characterizing boolean matrices with $\gamma_2$ norm at most $1$.
We call an
$m\times n$ boolean matrix $B$ \emph{blocky} if there exist disjoint sets $S_1,\ldots, S_k\subseteq [m]$
and disjoint sets $T_1,\ldots,T_k\subseteq [n]$ such that the support of $B$ is exactly $\bigcup_{i=1}^k S_i\times T_i$.
Simple examples of blocky matrices are zero matrices, all-$1$ matrices, and identity matrices. It turns out that blocky matrices are exactly the matrices with $\normgamma A\le 1$.

\begin{proposition}[{\rm \cite{liv95}}] \label{propcontractive}
A boolean matrix $A$ that satisfies $\normgamma A\le 1$ is a
blocky matrix. In particular, either $A = 0$ and $\normgamma A = 0$,
or $A$ is a nonzero blocky matrix and $\normgamma A = 1$.
\end{proposition}

It is natural, then, to wonder if every matrix of small $\gamma_2$ norm can be expressed as a sum of a bounded number of blocky matrices.

\begin{conjecture}[{\rm\cite{hhh2023}}, Conjecture {\small III}\/]\label{conjblocky}
Suppose that $A$ is a finite boolean matrix
with $\normgamma A\le \lambda$. Then we may write
\begin{equation}\label{eqblockyexpression}A = \sum_{i=1}^L \pm B_i,\end{equation}
where the $B_i$ are blocky matrices and $L$ depends only on $\lambda$.
\end{conjecture}

The converse to the conjecture clearly holds, since if $A = \sum_{i=1}^L \pm B_i$, then by the triangle inequality and \Cref{propcontractive}, $\normgamma A\le L$.
Since the $\gamma_2$ norm coincides with the Schur multiplier norm (this result is due to Grothendieck~\cite{grothendieck}, and a readable modern proof can be found in~\cite{pisier2001}), \Cref{conjblocky} is an analogue, in the algebra of Schur multipliers, of Green and Sanders's quantitative version~\cite{greensanders} of Cohen's idempotent theorem~\cite{cohen1960}. Indeed, given a boolean function $f:G \to \{0,1\}$ on a finite abelian group $G$, define $A :  G\times G \to \{0,1\}$ by $A(x,y) = f(x-y)$. Viewing $A$ as a boolean matrix,  we have
\begin{equation}\normgamma A = \normalg{f},\label{eqgammaspectral}\end{equation} 
where $\normalg{f}$ denote the Fourier algebra norm $\sum_{a\in \hat G} \bigl|\hat f(a)\bigr|$ 
(see, e.g., \cite[Corollary~3.13]{hhh2023} for a proof of this fact.) The Green--Sanders theorem states that one may write $f$ as a signed sum
\begin{equation}\label{eqgreensanders}
f = \sum_{i=1}^L \pm \one_{s_i + H_i}
\end{equation}
for $L$ depending only on $\lambda$ and some cosets $s_1 + H_1,\ldots,s_L + H_L$ of $G$.
The cosets given by the Green--Sanders theorem give rise to blocky matrices under this translation, and thus their theorem verifies a special case of \Cref{conjblocky} for matrices of the form $A(x,y)=f(x-y)$. This can be further generalized to nonabelian groups using a theorem of Sanders~\cite{sanders2011}. 

\Cref{conjblocky} is also equivalent (see \cite[Theorem~3.10]{hhh2023}) to the claim that any idempotent Schur multiplier on a countable domain can be written as a finite sum of contractive idempotents. This latter statement was originally conjectured by A.~Katavolos and V.~I.~Paulsen \cite{kp05} and is regarded as one of the challenging open problems in the area (see~\cite{MR2777487} and~\cite[Question 3.13]{elt16}). 

The minimum number of summands necessary in any decomposition~\eqref{eqblockyexpression} of $A$ into blocky matrices is called the \emph{block complexity} of $A$. \Cref{conjblocky} posits that every boolean matrix with constant $\gamma_2$ norm must have constant block complexity. Results concerning the block complexity of certain families of matrices, as well as random matrices, are presented in~\cite{avrahamyehudayoff}. In addition, block complexity has found applications in circuit complexity~\cite{williams,avrahamyehudayoff,hhh2023,hh2024}. 

Two advances towards the resolution of \Cref{conjblocky} have very recently been made. The first is due to the authors, and states that the block complexity of a matrix with bounded $\gamma_2$ norm is at most polylogarithmic as a function of the matrix's dimension.

\begin{theorem}[{\rm\cite{blockypolylog2025}}, Theorem 1.3\/]
Suppose that $A$ is an $n\times n$ boolean matrix
with $\normgamma A\le \lambda$. Then we may write
\begin{equation}A = \sum_{i=1}^L \pm B_i,\end{equation}
where the $B_i$ are blocky matrices and $L = 2^{O(\gamma^7)} \ln(n)^2$.
\end{theorem}

The other recent stride towards \Cref{conjblocky} was made by I.~Balla, L.~Hambardzumyan, and I.~Tomon. Their result relies on the following definition.
A \emph{monochromatic rectangle} in an integer matrix $A$ is the product of a subset $S$ of rows
and a subset $T$ of columns such that the entries $A(i,j)$ are all the same
for $(i,j)\in S\times T$.
We shall call a monochromatic rectangle all of whose entries are $b$ a \emph{$b$-rectangle}.

\begin{theorem}[{\rm\cite{bht2025}}, Theorem 1.1\/]\label{thmbht}
Suppose that $A$ is an $m\times n$ boolean matrix with
$\normgamma A\le \lambda$. There is a monochromatic rectangle $S\times T$ in $A$, where $S\subseteq [m]$ and $T\subseteq [n]$ satisfy
\begin{equation}\frac{|S|\cdot|T|}{mn} \ge 2^{-O(\lambda^3)}.\end{equation}
Specifically, if more than half of $A$'s entries are $1$, then $S\times T$ is a $1$-rectangle, and otherwise it is a $0$-rectangle.
\end{theorem}

A proof that \Cref{thmbht} would follow (without any specific quantitative bound) from \Cref{conjblocky} is presented as \cite[Lemma~3.5]{hhh2023}.

The density of $A$ determines whether \Cref{thmbht} furnishes a $0$-rectangle or a $1$-rectangle. Thus, given a blocky matrix such as the identity matrix, the theorem simply identifies a $0$-rectangle and one gains little information regarding the structure of $1$-rectangles in the matrix.

The main theorem of our present paper, whose proof relies on \Cref{thmbht},  shows that in any boolean matrix of bounded $\gamma_2$ norm, one can pick out a constant fraction of the $1$-entries to form a blocky matrix.

\begin{theorem}[Main theorem]\label{thmmain}
Let $A$ be an $m\times n$ boolean matrix $A$ with $\normgamma A\le \lambda$ in which the number of $1$-entries is $F$. There exists an $m\times n$ blocky matrix $B$ containing at least $F/2^{2^{O(\lambda)}}$ $1$-entries, such that, for all $(i,j)\in [m]\times [n]$, $B(i,j) = 1$ only if $A(i,j) = 1$.
\end{theorem}

To compare our theorem to \Cref{thmbht}, it is perhaps illuminating to narrow our focus once again to the case where $A(x,y) = f(x-y)$ for some boolean function $f : \ZZ_2^n \to \{0,1\}$, where $\normgamma A = \norm{f}_A\le \lambda$. In this setting, a theorem of A.~Shpilka, A.~Tal, and B.~Volk~\cite{shpilkatalvolk} states that there must exist some affine subspace $V\subseteq \ZZ_2^n$ of codimension at most $\lambda^2$ such that $f$ is constant on $V$. One can view \Cref{thmbht} of Balla, Hambardzumyan, and Tomon as an analogue of this theorem, as it guarantees the existence of a large submatrix on which the matrix is constant. 

On the other hand, the Green--Sanders decomposition~\eqref{eqgreensanders} implies that $f$ must be identically equal to $1$ on an affine subspace $V \subseteq \ZZ_2^n$ whose size is at least a constant fraction (depending on $\lambda$) of the support of $f$.  \Cref{thmmain} can be viewed as the matrix analogue of this latter statement. 

\Cref{thmmain} can also be regarded as a weak stability result for \Cref{propcontractive}, in the sense that one can still expose some degree of block substructure in $A$ even after relaxing the condition $\normgamma A\le 1$. 

The following theorem appears in~\cite{bht2025}, though without quantitative bounds and with a proof based on a recent graph-theoretic result~\cite{hmst2025}. Here, we derive an explicit double-exponential bound as a straightforward corollary of \Cref{thmmain}.

\begin{corollary}[{\rm\cite{bht2025}}, Theorem 1.6\/]\label{thmbhtfive}
If $A$ is an $m\times n$ matrix with $\normgamma A \le \lambda$ in which the number of $1$-entries is $F$, then there is a $1$-rectangle in $A$ of dimension $s\times t$, where $s \ge c_\lambda F/n $ and $t\ge c_\lambda F/m$ for some $c_\lambda \ge 1/2^{2^{O(\lambda)}}$.  
\end{corollary}
\begin{proof}
\Cref{thmmain} gives us a blocky matrix $B$ with at least $c_\lambda' F$ many $1$-entries for some constant $c_\lambda'\ge 1/2^{2^{O(\lambda)}}$. Let $\{S_i\times T_i\}_{i=1}^k$ denote the rectangles on which $B$ is supported (so the sets $S_i\subseteq[m]$ are all disjoint, as are the sets $T_i\subseteq[n]$). Let $I$ be the set of $i\in [k]$ with $|S_i|< c_\lambda' F/(2n)$. We have
\begin{equation}\sum_{i\in I} |S_i\times T_i| < {c_\lambda' F\over 2n}\sum_{i\in I} |T_i| \le {c_\lambda' F\over 2}.\end{equation}
Likewise, for the set $J$ of all $i\in [k]$ with $|T_i|< c_\lambda' F/(2m)$, we see that
\begin{equation}\sum_{i\in J} |S_i\times T_i| < {c_\lambda' F \over 2m}\sum_{i\in J} |S_i| \le {c_\lambda' F\over 2}.\end{equation} 
Since
\begin{equation}\sum_{i=1}^k |S_i\times T_i| = c_\lambda' F,\end{equation}
we have
\begin{equation}\sum_{i\in [k]\setminus(I\cup J)} |S_i\times T_i| > 0,\end{equation}
and consequently, there exists $i$ such that $|S_i| \ge c_\lambda' F/(2n)$ and $|T_i| \ge c_\lambda' F/(2m)$. The corollary follows by taking $c_\lambda = c_\lambda'/2$.
\end{proof}

In \Cref{sec:Potential}, we introduce the key parameter of our proof, which we refer to as the potential function.   In \Cref{sec:Threshold}, we discuss the threshold dimension of a matrix, a well-known parameter closely connected to the notion of stability in model theory. The potential function and the threshold dimension are the two main ingredients in the proof of \Cref{thmmain}. Our approach proceeds by identifying submatrices in which these parameters decrease, ultimately leading to the discovery of large 1-monochromatic rectangles that comprise the final blocky matrix.  The proof of \Cref{thmmain} is presented in \Cref{sec:Proof}.


\section{The potential function}
\label{sec:Potential}

Every $m\times n$ boolean matrix $A$ corresponds to a
bipartite graph with bipartition $([m],[n])$. In this context, the sets 
\begin{equation}R_i(A) = \bigl\{ j\in [n] : A(i,j) = 1\bigr\}\qquad\hbox{and}\qquad
C_j(A)= \bigl\{ i\in [m] : A(i,j) = 1\bigr\}\end{equation}
are the neighbourhoods corresponding to vertices $i\in [m]$ and $j\in [n]$ in the graph.
In most situations, there will only be one matrix under scrutiny, so we will simply write $R_i$ and $C_j$ when the choice of matrix $A$ is obvious from context.

Let $A$ be a matrix with $\normgamma A\le \lambda$.
By rescaling the vectors given by the definition of factorization norm, we may express $A$ as a product $A = UV$ for some real matrices $U$ and $V$ such that
\begin{equation*}\max\bigl( \norm{u_1}_2,\ldots,\norm{u_m}_2\bigr) \le 1\end{equation*}
and
\begin{equation*}\max\bigl(\norm{v_1}_2,\ldots,\norm{v_n}_2\bigr) \le \lambda,\end{equation*}
where $u_1,\ldots, u_m$ are the rows of $U$ and $v_1,\ldots, v_n$ are the columns of $V$. We shall call such a factorization of $A$ a \emph{$\lambda$-factorization}. If $U$ is $m\times t$ and $V$ is $t\times n$ for some integer $t$, note that the set $\{u_1, \ldots, u_m, v_1, \ldots, v_n\}$ of vectors is contained in a $t'$-dimensional subspace of $\RR^t$ for some $t'\le m+n$. So, without loss of generality, we always assume that all the vectors in a $\lambda$-factorization of an $m\times n$ matrix belong to $\RR^{m+n}$.

For any matrix $A$ and $\lambda$-factorization $A = UV$, define
\begin{equation}\Pi_{U,V}(A) =  \sum_{s=1}^m \norm{u_s}_2 |R_s(A)|.\end{equation}
We observe the following relation between this parameter and the \emph{Frobenius norm}, which is defined by
\begin{equation} \normF{A}^2 = \sum_{i=1}^m \sum_{j=1}^n \bigl|A(i,j)\bigl|^2 = \tr(AA^*).\end{equation}
If the matrix $A$ is boolean, $\normF{A}^2$ simply counts the number of $1$-entries in $A$.

\begin{proposition}\label{proppotentialbounds}
Let $A$ be a boolean $m\times n$ matrix with a $\lambda$-factorization $A = UV$. Then
\begin{equation} {\normF{A}^2\over \lambda^2} \le \Pi_{U,V}(A) \le \normF{A}^2,\end{equation}
with equality in the lower bound only if $A$ is a blocky matrix.
\end{proposition}
\begin{proof}The upper bound follows from the fact that $\normF{A}^2 = \sum_{i=1}^m |R_i|$. For the lower bound, observe that
\begin{equation}
\label{eq:potential_lower}
\sum_{i=1}^m \norm{u_i}_2^2 |R_i|
\ge \sum_{i=1}^m \sum_{j\in R_i} \norm{u_i}_2^2 \cdot {\norm{v_j}_2^2\over \lambda^2}
\ge {1\over \lambda^2} \sum_{i=1}^m \sum_{j\in R_i} \bigl|\langle u_i, v_j\rangle\bigr|^2 = {1\over \lambda^2} \normF{A}^2,\end{equation}
by the Cauchy--Schwarz inequality.
If equality holds in this lower bound, then by \eqref{eq:potential_lower} we have $\norm{u_i}_2^2 = 1/\lambda^2$ for all nonzero rows $i$. This means that $\normgamma A \le 1$, and $A$ must be blocky, by \Cref{propcontractive}.
\end{proof}

It shall be convenient for us to shed the dependence on a factorization $A= UV$, so for a matrix $A$ with $\normgamma A \le \lambda $, we define the \emph{potential} of $A$ to be the parameter
\begin{equation}\Pi_\lambda(A) = \inf_{A = UV} \Pi_{U,V}(A),\end{equation}
where the infimum runs over all $\lambda$-factorizations $UV$ of $A$.
The following proposition shows that this infimum is attained.

\begin{proposition}\label{propinfimum}
Let $A$ be an $m\times n$ boolean matrix with $\normgamma A \le \lambda$.
Then there exists a $\lambda$-factorization $A = UV$ of $A$ such that
$\Pi_\lambda(A) = \Pi_{U,V}(A)$.
\end{proposition}

\begin{proof} Let $t = m+n$, and for positive integers $r$ and $s$ let $\RR^{r\times s}$ denote the space of all $r\times s$ real matrices. The set of all $m\times t$ matrices $U$ with all rows $u_i$ having $\norm{u_i}_2 \le 1$ is a compact subset of $\RR^{m\times t}$, and likewise the set of all $t\times n$ matrices $V$ with all columns $v_j$ having $\norm{v_j}_2\le \lambda$ is a compact subset of $\RR^{t\times n}$. Hence the set of all $(U,V)$ with $UV = A$ is a compact subset of $\RR^{m\times t}\times \RR^{t\times n}$. The function $\sum_{i=1}^m \norm{u_i}_2^2 |R_i|$ is continuous on $\RR^{m\times t} \times \RR^{t\times n}$, so it attains its infimum on this compact set.
\end{proof}

Note also that if $A'$ is a submatrix of $A$ with $\normgamma A \le \lambda$, then  $\Pi_\lambda(A')\le \Pi_\lambda(A)$.




Our goal in the remainder of this section is to find a set of columns in $A$ whose deletion results in a sizable drop in potential.
We begin by observing the following proposition, which follows from the H\"older inequality for Schatten norms.

\begin{proposition}\label{proptracecalc}
Let $X$ be an $m\times t$ matrix and $Y$ a $t\times n$ matrix. We have
\begin{equation}\normF{XY}^2 \le \normF{XX^\T} \cdot \normF{Y^\T Y}.\end{equation}
\end{proposition}

\begin{proof}
Since the trace is invariant under cyclic permutation of factors, we have
\begin{equation*}\normF{XY}^2 = \tr(Y^\T X^\T XY)
= \tr(X^\T XY Y^\T) \le \normF{X^\T X}\cdot \normF{Y Y^\T} = \normF{X X^\T}\cdot \normF{Y^\T Y}.\qedhere\end{equation*}
\end{proof}

The following lemma shows that for every $i\in [m]$, the average inner product between two columns of $V$ indexed by elements of $R_i$ is bounded away from zero.

\begin{lemma}\label{lemboundedaway}
Let $A$ be an $m\times n$ boolean matrix with $\lambda$-factorization $A = UV$.
We have
\begin{equation}|R_i|^2 \le \sum_{r,s\in R_i}
\bigl|\langle v_r, v_s\rangle\bigr|^2 \le \lambda^4 |R_i|^2.\end{equation}
for all $1\le i\le m$, and
\begin{equation}{|C_j|^2\over \lambda^2} \le \sum_{r,s\in C_j}
\bigl|\langle u_r, u_s\rangle\bigr|^2 \le |C_j|^2\end{equation}
for all $1\le j\le n$.
\end{lemma}

\begin{proof}
The upper bounds are easy, since we can bound
\begin{equation}\bigl| \langle v_r, v_s\rangle \bigr|^2 \le \norm{v_r}_2^2 \cdot \norm{v_s}_2^2 \le \lambda^4\end{equation}
for all $r,s\in [n]$, as well as
\begin{equation}\bigl| \langle u_r, u_s\rangle \bigr|^2 \le \norm{u_r}_2^2 \cdot \norm{u_s}_2^2 \le 1\end{equation}
for all $r,s\in[m]$.

For the first lower bound, let $1\le i\le [m]$ be given, let $X$ be the row vector $u_i$ and let $Y$ be the matrix $V$ but with column $j$ set to zero if $A(i,j)$ equals $0$. This construction gives
\begin{equation}\normF{XX^\T}^2 = \norm{u_i}_2^2 \le 1\end{equation}
and
\begin{equation}\normF{Y^\T Y}^2 = \sum_{r,s\in R_i} \bigl|\langle v_r, v_s\rangle\bigr|^2.\end{equation}
On the other hand, we have
\begin{equation}\normF{XY}^2 = \sum_{j=1}^n \langle u_i, v_j\rangle^2 = \sum_{j=1}^n A(i,j) = |R_i|,\end{equation}
to which we may apply \Cref{proptracecalc}, yielding
\begin{equation} |R_i|\le  \normF{XY}^2 \le \normF{X X^\T} \cdot \normF{Y^\T Y} \le \biggl(\sum_{r,s\in R_i} \langle v_r, v_s\rangle^2\biggr)^{1/2}.\end{equation}
Squaring gives us our desired bound, and the second lower bound is proved in the exact same manner.\end{proof}

A corollary of \Cref{lemboundedaway} is that for every column $t$, there are a significant number of pairs $(u_r, u_s)$ of rows of $U$ with $r$ and $s$ both in $C_t$ and $\bigl|\langle u_r, u_s\rangle\bigr|$ fairly large.

\begin{corollary}
\label{cor:large_inp_pairs}
Let $A$ be an $m\times n$ boolean matrix with $\lambda$-factorization $A = UV$.
For all $1\le t\le n$, we have the lower bound
\begin{equation}\biggl|\biggl\{(r,s)\in {C_t}^2 : \bigl|\langle u_r, u_s\rangle\bigr|^2 \ge {1\over 2\lambda^2}\biggr\}\biggr|\ge {|C_t|^2\over 2\lambda^2}.\end{equation}
\end{corollary}
\begin{proof} Let $D\subseteq {C_t}^2$ be the set of $(r,s)$ with $\bigl| \langle u_r, u_s\rangle\bigr|^2 \ge 1/(2\lambda^2)$. By \Cref{lemboundedaway},
\begin{equation}{|C_t|^2\over \lambda^2} \le \sum_{r,s\in C_t} \bigl| \langle u_r, u_s\rangle\bigr|^2 \le \sum_{(r,s)\in D} \bigl| \langle u_r, u_s\rangle\bigr|^2 + {|C_t|^2\over 2\lambda^2},\end{equation}
and the corollary follows, since $\bigl| \langle u_r, u_s\rangle \bigr|^2 \le 1$ for all $r,s\in [m]$.
\end{proof}

For every row $i \in [m]$, define 
\begin{equation}
\Delta_i= \biggl\{ s\in [m] : \bigl| \langle u_i, u_s\rangle\bigr|^2\ge {1\over 2\lambda^2}\biggr\}.
\end{equation}

Given subsets $S \subseteq [m]$ and $T \subseteq [n]$, we write $A_{S \times T}$ to denote the submatrix of $A$ restricted to rows in $S$ and columns in $T$. The following lemma establishes the existence of a row $i$ for which a significant fraction of the $1$-entries in $A_{[m] \times R_i}$ are located within  $A_{\Delta_i \times R_i}$.


\begin{lemma}[Detecting a pivotal row\/]\label{lemusefuli}
Let $A$ be an $m\times n$ boolean matrix with $\lambda$-factorization $A = UV$.
There exists $1\le i\le m$ such that
\begin{equation}
\normF{A_{\Delta_i \times R_i}}^2 \ge {1\over 2\lambda^2} \normF{A_{[m] \times R_i}}^2.\end{equation}
\end{lemma}

\begin{proof} \Cref{cor:large_inp_pairs} tells us that
\begin{equation}
{1\over 2\lambda^2} \sum_{t=1}^n |C_t|^2 \le \sum_{t=1}^n \sum_{i=1}^m \sum_{s=1}^m \one_{[i,s\in C_t]} \one_{[s\in \Delta_i]}
\le \sum_{i=1}^m \sum_{t\in R_i} \sum_{s=1}^m \one_{[s\in C_t\cap \Delta_i]}
= \sum_{i=1}^m  \normF{A_{\Delta_i \times R_i}}^2.
\end{equation}
On the other hand,
\begin{equation}\sum_{t=1}^n |C_t|^2 = \sum_{i=1}^m \sum_{t\in R_i} |C_t|=\sum_{i=1}^m \normF{A_{[m] \times R_i}}^2,\end{equation}
whence the claim follows.
\end{proof}

The row $i$ identified by \cref{lemusefuli} will play a pivotal role in the proof of our main theorem. The proof is divided into several cases based on the properties of this row, and in each case, we demonstrate how to advance toward constructing a large blocky matrix.

We now prove the principal lemma of this section, which shows that if $\normF{A_{\Delta_i \times [n]}}^2$ is significantly larger than  $\normF{A_{[m] \times R_i}}^2$, then we can restrict to $[m] \times R_i^c$ as the loss of $1$-entries from this restriction is more than compensated for by the decrease in the potential function.

\begin{lemma}\label{lempotential}
Let $A$ be an $m\times n$ boolean matrix with $\normgamma{A} \le \lambda$, and suppose that $i\in [m]$ is a row with $R_i\ne \emptyset$ and
\begin{equation}\normF{A_{\Delta_i \times [n]}}^2  \ge 4\lambda^2 \normF{A_{[m] \times R_i}}^2.\end{equation}
The submatrix $A'=A_{[m] \times  R_i^c}$  satisfies
\begin{equation}\Pi_{\lambda}(A) - \Pi_{\lambda}(A') \ge 2\bigl(\normF{A}^2 - \normF{A'}^2\bigr).\end{equation}
\end{lemma}

\begin{proof} In this proof, parameters such as $R_s$ and $C_t$ will pertain to $A$ (and not $A'$) unless otherwise specified. 

Consider a $\lambda$-factorization $A = UV$  achieving the potential, i.e.,  $\Pi_{U,V}(A)=\Pi_\lambda(A)$. Define $U'$ to be the matrix obtained by replacing every row $u_s$ of $U$  with its component orthogonal to $u_i$: 
\begin{equation}
u_s' = u_s - {\langle u_i, u_s\rangle \over \norm{u_i}_2^2}u_i.
\end{equation}
Let $V'$ be the submatrix of $V$ by removing the columns $\{v_t : t\in R_i\}$. 

Since $\langle u_i, v_t\rangle = 0$ for all $t\notin R_i$, for such $t$, we have
\begin{equation}
\langle u_s, v_t\rangle = \biggl\langle u_s' + {\langle u_i, u_s\rangle \over \norm{u_i}_2}u_i, v_t\biggr\rangle = \langle u_s', v_t\rangle,
\end{equation}
and therefore, $A'=U'V'$ is a $\lambda$-factorization of $A'$.

Let $\eta= \sum_{t\in R_i} |C_t|$, so that $\normF{A'}^2 = \normF{A}^2-\eta$. To bound the potential of $A'$, note that for all $s\in \Delta_i$, we have
\begin{equation}\norm{u_s}_2^2 - \norm{u_s'}_2^2 = \bigl|\langle u_i, u_s \rangle\bigr|^2 \ge {1\over 2\lambda^2}.\end{equation}
This and our assumption on $i$ together imply 
\begin{equation}\Pi_{U,V}(A) - \Pi_{U',V'}(A')
= \sum_{s=1}^m \Bigl(\norm{u_s}_2^2 |R_s|
- \norm{u_s'}_2^2 \bigl| R_s(A')\bigr|\Bigr)
\ge {1\over 2\lambda^2} \sum_{s\in \Delta_i} \norm{u_s}_2^2 |R_s|
\ge 2 \eta ,\end{equation}
which completes the proof.
\end{proof}

\section{The threshold dimension}
\label{sec:Threshold}

The \emph{threshold dimension} of a boolean matrix $A$, denoted $\TD(A)$, is the largest positive integer $d$ such that there exist $i_1,\ldots, i_d\in [m]$ and $j_1,\ldots, j_d\in [n]$ such that
$A(i_s, j_t) = \one_{[s\ge t]}$ for all $s,t\in [d]$.
This definition is related to the notion of stability of M.~Malliaris and S.~Shelah~\cite{malliarisshelah}, in which one forbids half-graphs.

It is clear that restricting to a submatrix cannot cause the threshold dimension to increase.
Note that a boolean matrix $A$ is blocky if and only if it does not contain (up to permutation of rows and columns) the submatrix of the form  $\matr1011$. This is equivalent to the condition $\TD(A) = 1$.

We have the following relation between the $\gamma_2$ norm and the threshold dimension. This proposition is well known, following from analogous bounds for the margin~\cite{fsss2003} and discrepancy~\cite{srinivasanyehudayoff} of a matrix. For completeness, we present a short proof here in the language of the $\gamma_2$ norm.

\begin{proposition}
Every matrix $A$ with $\normgamma A\le \lambda$ satisfies $\TD(A)\le 2^{O(\lambda)}$.
\end{proposition}

\begin{proof} Let $n$ be a positive integer and let $G$ be the boolean $n\times n$ matrix with $G(i,j) = \one_{[i\ge j]}$ for all $i,j\in [n]$. We are done if we can show that $\normgamma G\ge \Omega(\log_2 n)$.

Let $S = \{0,\ldots,n-1\} \subseteq \ZZ_{2n}$. We define $H : \ZZ_{2n}\times \ZZ_{2n} \to \{0,1\}$ by $H(x,y) = \one_S(x-y)$. As a matrix, $H$ takes the form
\begin{equation} H = \begin{pmatrix} G & J-G \cr J-G & G\end{pmatrix}
= \begin{pmatrix} G & G \cr G & G \end{pmatrix} \circ \begin{pmatrix} J & -J \cr -J & J\end{pmatrix} + \begin{pmatrix} 0 & J \cr J & 0 \end{pmatrix},\end{equation}
where $J$ denotes the $n\times n$ all-$1$ matrix and $\circ$ denotes the entrywise (Schur) product. Since
\begin{equation*}\begin{pmatrix} 1 & -1 \cr -1 & 1\end{pmatrix}=
\begin{pmatrix} 1\cr -1\end{pmatrix}
\begin{pmatrix} 1 & -1\end{pmatrix}\end{equation*}
and the $\gamma_2$ norm is invariant under duplication of rows and columns and satisfies $\normgamma{A \circ B} \le \normgamma{A}\normgamma{B}$, we have
\begin{equation} \normgamma H \le
\normgamma{\begin{pmatrix} G & G \cr G & G\end{pmatrix}} \cdot \normgamma{\begin{pmatrix}J & -J \cr -J & J\end{pmatrix}} + \normgamma{\begin{pmatrix}0 & J \cr J & 0\end{pmatrix}}
= \normgamma G + 1.\end{equation}
On the other hand, since $H(x,y) = \one_S(x-y)$, by~\eqref{eqgammaspectral} we have
\begin{equation*}\normgamma H = \sum_{a\in \hat{\ZZ_{2n}}} \bigl| \hat{\one_S}(a)\bigr|.\end{equation*}
Let $e(x)$ denote $e^{2\pi i x}$ for short. We compute
\begin{equation}\hat{\one_S}(a)
= {1\over 2n} \sum_{x=0}^{n-1} e\Bigl({-ax\over 2n}\Bigr)
= {1\over 2n} \cdot {1-e(-a/2) \over 1- e\bigl( -a/(2n)\bigr)}
= {\one_{[a\,\hbox{\scriptsize odd}]}\over n}\cdot  {1\over 1- e\bigl(-a/(2n)\bigr)}.\end{equation}
So consider any odd $a\le n$. We have
\begin{equation}\bigl|\hat{\one_S}(a)\bigr|
= {1\over n} \Biggl| {1\over 1-e\bigl( -a/(2n)\bigr)}\Biggr|
= {1\over 2n\bigl| \sin\bigl(-a\pi/(2n)\bigr)\bigr|},\end{equation}
and since $|\sin x| \le |x|$ for all $x$, we can bound
$\bigl| \hat{\one_S}(a)\bigr| \ge 1/(\pi a)$.
\begin{equation}\normgamma H = \sum_{a\in \hat{\ZZ_{2n}}} \bigl|\hat{\one_S}(a)\bigr|
\ge \sum_{a=0}^{n-1} {\one_{[a\,\hbox{\scriptsize odd}]} \over \pi a}
= {1\over 2\pi} \ln(n) - O(1).\end{equation}
The result follows.\end{proof}

Let $i$ be the pivotal row provided by \cref{lemusefuli}. As discussed in the previous section, if this row satisfies the condition in \cref{lempotential}, then we can make progress toward uncovering the blocky structure by removing the columns in  $R_i$. Otherwise, we find ourselves in the case where
\begin{equation}{1\over 4\lambda^2} 
\normF{A_{\Delta_i \times [n]}}^2  \le \normF{A_{[m]\times R_i}}^2  \le 2\lambda^2 \normF{A_{\Delta_i \times R_i}}^2.\end{equation}
In this case, the number of $1$-entries inside $\Delta_i \times R_i$ is comparable to the number of $1$-entries in $(\Delta_i \times [n]) \cup ([m] \times R_i)$. Consequently, if we can identify a large blocky structure within the submatrix $A_{\Delta_i \times R_i}$,  we may safely include it in the final blocky matrix in $A$, since all the 1-entries lost in the process are confined to $(\Delta_i \times [n]) \cup ([m] \times R_i)$. 
In this section, we show that under this alternative, a large blocky submatrix can indeed be extracted from  $A_{\Delta_i \times R_i}$ by analyzing its threshold dimension.

\begin{lemma}\label{lemtd}
Let $A$ be an $m\times n$ boolean matrix with $\normgamma A \le \lambda$ and $\TD(A) \le d$ for some $\lambda,d\ge 1$. Suppose further that there is a row $i\in [m]$ with
\begin{equation}
\label{eq:assumption_useful}
{1\over 4\lambda^2} 
\normF{A_{\Delta_i \times [n]}}^2  \le \normF{A_{[m]\times R_i}}^2  \le 2\lambda^2 \normF{A_{\Delta_i \times R_i}}^2.\end{equation}
Then 
\begin{equation}
\label{eq:whatisleft}
\bignormF{A_{\Delta_i^c \times R_i^c}}^2 \ge \normF{A}^2 - 10\lambda^4 \bignormF{A_{\Delta_i \times R_i}}^2,\end{equation}
and moreover either
\begin{itemize}\setlength\itemsep{-2pt}
\item[i)] $A_{\Delta_i \times R_i}$ contains a monochromatic $1$-rectangle with at least $e^{-O(\lambda^3)} \bignormF{A_{\Delta_i \times R_i}}^2$ entries; or
\item[ii)] there exists a subset $S \subseteq \Delta_i$ such that the submatrix $A_{S \times R_i}$ contains at least half of the $1$-entries of $A_{\Delta_i \times R_i}$,  and its threshold dimension satisfies $\TD(A_{S \times R_i}) \le d-1$.
\end{itemize}
\end{lemma}

\begin{proof}
By \eqref{eq:assumption_useful}, deleting all rows in $\Delta_i$ and all columns in $R_i$ results in the support of $A$ decreasing by no more than
\[\normF{A_{\Delta_i \times [n]}}^2 + \normF{A_{[m]\times R_i}}^2 \le (8\lambda^4+2\lambda^2) \normF{A_{[m]\times R_i}}^2   \le 10 \lambda^4  \normF{A_{\Delta_i \times R_i}}^2, \] 
which verifies \eqref{eq:whatisleft}. 
Let
\begin{equation}j = \argmin_{t\in R_i} |C_t\cap \Delta_i|;\end{equation}
there are two cases.

If $(C_j\cap \Delta_i)\times R_i$ contains more than half of the $1$-entries in $ \Delta_i \times R_i$, then we let $S = C_j \cap \Delta_i$ and by our choice of $j$, we have
\begin{equation}
\normF{A_{S \times R_i}}^2 \ge {1\over 2}\normF{A_{\Delta_i \times R_i}}^2  \ge {1\over 2} \sum_{t\in R_i} |S| = {|S|\cdot |R_i|\over 2},\end{equation}
whence by \Cref{thmbht}, $A_{S \times R_i}$ must contain a $1$-rectangle of size at least $e^{-O(\lambda^3)} |S|\cdot |R_i|$, and the simple observation
\begin{equation}\bignormF{A_{\Delta_i \times R_i}}^2 \le 2\bignormF{A_{S \times R_i}}^2 \le 2|S|\cdot |R_i|\end{equation}
allows us to conclude that alternative (i) above holds.

If, on the other hand, $(C_j\cap \Delta_i)\times R_i$ contains at most half of the $1$-entries in $\Delta_i\times R_i$, then we let $S = \Delta_i \setminus C_j$, and we still have $\normF{A_{S \times R_i}}^2 \ge \bignormF{A_{\Delta_i \times R_i}}^2/2$. We are done if we can show that $\TD(A_{S \times R_i}) \le d-1$.

Suppose, towards a contradiction, that $\TD(A_{S \times R_i}) \ge d$. Then there exist $i_1, \ldots, i_d \in S$ and $j_1,\ldots, j_d\in R_i$ such that $A'(i_s, j_t) = \one_{[s\ge t]}$ for all $s,t\in [d]$. Note that $i_s\ne i$ for all $s\in [d]$, since $i\in C_j$, and $j_t\ne j$ for all $t\in [d]$, since the column indexed by $j_0$ contains only zeroes in $A_{S \times R_i}$. Therefore, we are able to set $i_{d+1} = i$ and $j_{d+1} = j$, and $A(i_s, j_t) = \one_{[s\ge t]}$ for all $s,t\in [d+1]$, and $\TD(A)\ge d+1$, contradicting our standing assumption on $A$.\end{proof}

\section{Proof of the main theorem}
\label{sec:Proof}
We now prove \Cref{thmmain} by combining  \Cref{lempotential} and \Cref{lemtd} in an inductive argument.

\begin{proof}[Proof of~\Cref{thmmain}]
Let $b(F,\Pi, d)$ be such that for every matrix $A$ with $\normgamma A\le \lambda$, $\normF{A}^2 = F$, $\Pi_\lambda(A) \le \Pi$, and $\TD(A) \le d$, there exists a blocky matrix $B$ of the same dimensions as $A$ with $\normF{B}^2 \ge b(F,\Pi, d)$ and $B(i,j) = 1$ only if $A(i,j) = 1$. Note that restricting to a submatrix cannot cause any of the parameters $F$, $\Pi$, and $d$ to increase. We shall show, by induction on these three parameters, that
\begin{equation}b(F, \Pi, d) \ge {F-\Pi/2 \over e^{O(\lambda^3)} (40\lambda^4)^d},\end{equation}
which proves the theorem, since
\begin{equation} {F-\Pi/2 \over e^{O(\lambda^3)} (40\lambda^4)^d}
\ge {F \over e^{O(\lambda^3)} \lambda^{2^{O(\lambda)}}}
= {F\over 2^{2^{O(\lambda)}}}.\end{equation}

There are three base cases. If $F = 1$, then $\Pi\le 1$, and we have
\begin{equation}b(1,\Pi, d) = 1 \ge {1/2\over e^{O(\lambda^3)} (40\lambda^4)^d}
\ge {1-\Pi/2\over  e^{O(\lambda^3)} (40\lambda^4)^d}\end{equation}
for all positive integers $d$, since a matrix with only a single $1$-entry is clearly blocky.
The smallest that $\Pi$ can be is $F/\lambda^2$, in which case the matrix is also blocky, hence for all positive integers $F$ and $d$ we have
\begin{equation}b(F,F/\lambda^2, d) = F \ge \Bigl(1-{1\over 2\lambda^2}\Bigr)F
\ge {F-\Pi/2\over  e^{O(\lambda^3)} (40\lambda^4)^d}.\end{equation}
Lastly, when $d=1$, the matrix is also forced to be blocky, and
\begin{equation}b(F,\Pi,1) = F \ge  {F-\Pi/2\over  e^{O(\lambda^3)} (40\lambda^4)}\end{equation}
for all $F$ and $\Pi$.

For the inductive step, let $A$ be an $m\times n$ boolean matrix with $\normgamma A\le \lambda$, $\normF{A}^2 = F$, $\Pi_\lambda(A) = \Pi$ and
$\TD(A) = d$. Let $A = UV$ be a $\lambda$-factorization such that $\Pi_{U,V}(A) = \Pi$, which exists by \Cref{propinfimum}. If there exists a row with $R_i\ne \emptyset$ and $\normF{A_{\Delta_i \times [n]}}^2  > 4\lambda^2 \normF{A_{[m] \times R_i}}^2$, then by \Cref{lempotential}, the submatrix $A_{[m] \times  R_i^c}$ satisfies $\normF{A_{[m] \times  R_i^c}}^2 = F - k$ and $\Pi(A_{[m] \times  R_i^c}) \le \Pi - 2k$ for some positive integer $k$. Thus we have
\begin{equation}b(F, \Pi, d) \ge b(F-k, \Pi-2k, d),\end{equation}
and by the induction hypothesis,
\begin{equation}b(F-k, \Pi-2k, d)
\ge {F-k - (\Pi-2k)/2\over e^{O(\lambda^3)}(40\lambda^4)^d}
\ge {F-\Pi/2 \over e^{O(\lambda^3)}(40\lambda^4)^d}.\end{equation}

If there does not exist a nonzero row with $\normF{A_{\Delta_i \times [n]}}^2  > 4\lambda^2 \normF{A_{[m] \times R_i}}^2$, then letting $i$ be the pivotal row given by \Cref{lemusefuli}, we obtain
\begin{equation}{1\over 4\lambda^2} 
\normF{A_{\Delta_i \times [n]}}^2  \le \normF{A_{[m]\times R_i}}^2  \le 2\lambda^2 \normF{A_{\Delta_i \times R_i}}^2,\end{equation}
and one of the two alternatives given by~\Cref{lemtd} holds. If the first alternative holds, then letting $k =\bignormF{A_{\Delta_i \times R_i}}^2$, the matrix $A_{\Delta_i \times R_i}$ contains a $1$-rectangle with $k/e^{O(\lambda^3)}$ entries, and meanwhile the matrix $A_{\Delta_i^c \times R_i^c}$ has support at least $F-10\lambda^4 k$. It follows that
\begin{equation}b(F,\Pi,d) \ge b(F-10\lambda^4 k, \Pi, d) + {k\over e^{O(\lambda^3)}}
\ge {F-10\lambda^2 k - \Pi/2 + (40\lambda^4)^d k \over e^{O(\lambda^3)} (40\lambda^4)^d}
\ge {F - \Pi/2 \over e^{O(\lambda^3)} (40\lambda^4)^d}.\end{equation}

In the other case given by \Cref{lemtd}, at least half of the $1$-entries of the matrix $A_{\Delta_i \times R_i}$ is in a submatrix $A'$ with $\TD(A') \le d-1$. Letting $k = \normF{A'}^2$, we also have
$k \ge \bignormF{A_{\Delta_i \times R_i}}^2/2$, so the support of $A_{\Delta_i^c \times R_i^c}$ is at least $F - 20\lambda^4 k$. Inductively adding up the supports of the blocky matrices in $A'$ and $A_{\Delta_i^c \times R_i^c}$ gives us
\begin{align}
b(F,\Pi,d)
&\ge b(k,k,d-1) + b(F-20\lambda^4 k, \Pi, d)\cr
&\ge {k-k/2\over e^{O(\lambda^3)} (40\lambda^4)^{d-1}} + {F-20\lambda^4 k - \Pi/2 \over e^{O(\lambda^3)} (40\lambda^4)^d}\cr
&\ge {F - \Pi/2 + (k/2)(40\lambda^4) - 20\lambda^4 k \over e^{O(\lambda^3)} (40\lambda^4)^d} \cr
&\ge {F - \Pi/2 \over e^{O(\lambda^3)} (40\lambda^4)^d}, \cr
\end{align}
which settles this case and completes the proof.\end{proof}

\section*{Acknowledgements}

We would like to thank Igor Balla, Lianna Hambardzumyan, and Istv\'an Tomon for sending us an early copy of their preprint~\cite{bht2025}, and Ben Cheung for many helpful discussions.
Both authors are funded by the Natural Sciences and Engineering Research Council of Canada.

\bibliographystyle{alphacitation}
\xpatchcmd{\em}{\itshape}{\slshape}{}{}
\bibliography{citations}
\end{document}